\documentclass[12pt]{article}
\usepackage[centertags]{amsmath}
\usepackage{amsfonts}
\usepackage{amssymb}
\usepackage{latexsym}
\usepackage{amsthm}
\usepackage{newlfont}
\usepackage{graphicx}
\usepackage{listings}
\usepackage{booktabs}
\usepackage{abstract}
\date{}

\newlength{\defbaselineskip}
\newcommand{\setlinespacing}[1]%
           {\setlength{\baselineskip}{#1 \defbaselineskip}}

\newcommand{\N}{{\mathbb{N}}}

\newcommand{\actaqed}{\hfill $\actabox$}
{\medskip\noindent \textit{Proof of #1. }}%
{\actaqed \medskip}

\def\R{{\mathbb R}}
\def\Z{\mathbb Z}

\def \<{\langle}
\def\>{\rangle}

\def \La{\Lambda}

\def \ff{\varphi}

\def\la{\lambda}

\def\bx{\mathbf x}
\def\by{\mathbf y}

\def\bk{\mathbf k}

\def\bW{\mathbf W}

\newtheorem{Theorem}{Theorem}[section]
\newtheorem{Lemma}{Lemma}[section]

\newtheorem{Corollary}{Corollary}[section]

\numberwithin{equation}{section}

\newcommand{\be}{\begin{equation}}
\newcommand{\ee}{\end{equation}}

\begin{document}

\title{ Numerical integration without smoothness assumption}
\author{V.N. Temlyakov\thanks{University of South Carolina, Lomonosov Moscow State University, and Steklov Institute of Mathematics.  }}
\maketitle
\begin{abstract}
{  We consider numerical integration in classes, for which we do not impose any smoothness assumptions. We illustrate how nonlinear approximation, in particular greedy approximation, allows us to guarantee some rate of decay of errors of numerical integration even in such a general setting with no smoothness assumptions. }
\end{abstract}

\section{Introduction} 
\label{I} 

The paper is devoted to numerical integration. The goal is to obtain rates of decay of errors 
of numerical integration for functions from a given function class. Theoretical aspects of the problem of numerical integration are intensely studied in approximation theory and in discrepancy theory. A typical problem in that regard is to study numerical integration in a given smoothness class (see, for instance, \cite{VTbookMA}, Ch. 6). It is a difficult area of research, related to discrepancy theory and other areas of research,  with a number of outstanding open problems (see, for instance, \cite{VT89}, \cite{VT164}, \cite{DTU}, and \cite{VT170}). 
In the case of classes of multivariate functions with mixed smoothness delicate number theoretical methods are used to build good cubature formulas. The main goal of this paper is to study numerical integration in much more general classes than smoothness classes. 
Clearly, we cannot expect that delicate methods developed for studying numerical integration in smoothness classes will apply to the case of general classes. It was observed in \cite{VT89} that very general method of nonlinear approximation, in particular greedy approximation, may be successfully used in numerical integration. However, it is known (see \cite{Tbook}, Section 2.7) that in general greedy approximation has a property of saturation. Usually, the saturation rate in $m^{-1/2}$, where $m$ is the number of iterations of a greedy algorithm. As a result in our applications of nonlinear approximation we cannot beat a barrier of $m^{-1/2}$. 

We now proceed to a detailed description of our results. Numerical integration seeks good ways of approximating an integral
$$
\int_\Omega f(\bx)d\mu
$$
by an expression of the form
\be\label{1.1}
\La_m(f,\xi) :=\sum_{j=1}^m\la_jf(\xi^j),\quad \xi=(\xi^1,\dots,\xi^m),\quad \xi^j \in \Omega,\quad j=1,\dots,m. 
\ee
It is clear that we must assume that $f$ is integrable and defined at the points
 $\xi^1,\dots,\xi^m$. Expression (\ref{1.1}) is called a {\it cubature formula} $(\xi,\La)$ (if $\Omega \subset \R^d$, $d\ge 2$) or a {\it quadrature formula} $(\xi,\La)$ (if $\Omega \subset \R$) with knots $\xi =(\xi^1,\dots,\xi^m)$ and weights $\La:=(\la_1,\dots,\la_m)$. 
 
 Some classes of cubature formulas are of special interest. For instance, the Quasi-Monte Carlo cubature formulas, which have equal weights $1/m$, are discussed in this paper. We use a special notation for these cubature formulas
 $$
 Q_m(f,\xi) :=\frac{1}{m}\sum_{j=1}^mf(\xi^j).
 $$
 
  For a function class $\bW$ we introduce a concept of error of the cubature formula $Q_m(\cdot,\xi)$ by
\be\label{1.3}
Q_m(\bW,\xi):= \sup_{f\in \bW} |\int_\Omega fd\mu -Q_m(f,\xi)|. 
\ee
The quantity $Q_m(\bW,\xi)$ is a classical characteristic of the quality of a given cubature formula $Q_m(\cdot,\xi)$.
 
 Let $1\le p\le\infty$ and let $F\in L_{p'}([0,1)^d)$ be a $1$-periodic function, where $p':= \frac{p}{p-1}$ is a dual to $p$ exponent. Consider the following class of functions
$$
\bW^F_p := \{f\, :\, f(\bx) = \int_{[0,1)^d} F(\bx-\by)\ff(\by)d\by,\quad \|\ff\|_p \le 1\}.
$$
Note that the case of classes $\bW^r_p$ of multivariate functions with bounded mixed derivative corresponds to the function 
$$
F(\bx):= F_{r}(\bx) := \prod_{j=1}^d F_{r}(x_j),\quad \bx=(x_1,\dots,x_d),
$$
where for a scalar $x$
$$
F_{r}(x):= 1+2\sum_{k=1}^\infty k^{-r}\cos (2\pi kx-r \pi/2).
$$

The following result (in a more general setting) is proved 
in \cite{VT149} (see also \cite{VT89} for previous results).  

\begin{Theorem}\label{IT1} Let $p\in (1,2]$ and let $\bW^F_p$ be a class of functions defined above. Assume that  $\|F\|_{p'}\le 1$. Then for any  $m$ there exists (provided by an appropriate greedy algorithm)  a cubature formula { $Q_m(\cdot,\xi)$} such that
$$
 Q_m(\bW^F_p,\xi)\le C(p-1)^{-1/2} m^{-1/2}  .
$$
\end{Theorem}

Proof of Theorem \ref{IT1} is based on the theory of greedy algorithms in Banach spaces. That theory is well developed under assumption that the Banach space is uniformly smooth, which means $\lim_{u\to 0} \rho(u)/u =0$, where $\rho(u)$ is a modulus of smoothness of the space. It is well known that the space $L_1$ is not uniformly smooth. This is why the case $p=1$ is excluded in Theorem \ref{IT1}. In this paper we analyze an algorithm -- the Averaging Search algorithm -- which allows us to prove an analog of Theorem \ref{IT1} in the case $p=1$ under additional assumptions on the kernel $F$.  We begin with the definition of the Averaging Search algorithm. This algorithm and its greedy version were analyzed in the recent paper \cite{BS}. 

{\bf Averaging Search algorithm.} Let $g\in L_1([0,1)^d)$ be a real $1$-periodic function satisfying  condition $\int_{[0,1)^d} g(\bx) d\bx =0$. We build a sequence $\xi^1$,..., $\xi^m$ of points from $[0,1)^d$ inductively. At the first step choose any $\xi^1\in [0,1)^d$. Suppose, $m\ge 2$ and after $m-1$ steps of the algorithm we have built points $\xi^1$,..., $\xi^{m-1}$. Then, at the $m$th step we choose $\xi^m\in [0,1)^d$ such that
\be\label{as1}
\sum_{j=1}^{m-1} g(\xi^m-\xi^j)  \le 0.
\ee

Note that such $\xi^m$ always exists. Indeed, by our assumption we have
$$
\int_{[0,1)^d} \sum_{j=1}^{m-1} g(\bx-\xi^j)d\bx  = 0
$$
and, therefore, there exists $\xi^m\in [0,1)^d$ satisfying (\ref{as1}). 

We now proceed to the main result of this paper. Denote $F^0(\bx):= F(\bx)-\hat F(\mathbf 0)$. 

\begin{Theorem}\label{asT1} Suppose that $\|F\|_\infty <\infty$ and function $F^0$ is a real 
even function, satisfying the condition $\hat F^0(\bk) \ge 0$, $\bk\in \Z^d$, $\bk\neq \mathbf 0$. Then for any $m\in\N$ there exists (provided by the Averaging Search algorithm applied to $g=F^0$) a set $\xi:=\{\xi^j\}_{j=1}^m$ of points in $[0,1)^d$ such that for the  cubature formula $Q_m(\cdot,\xi)$ we have
$$
Q_m(\bW^F_1,\xi) \le \|F^0\|_\infty m^{-1/2}.
$$

\end{Theorem}

Associate with a cubature formula $Q_m(\cdot,\xi)$ and the function $F$ the following function (see \cite{VT164})
\be\label{I1}
g_{\xi,Q,F}( \bx) := \sum_{ \bk\neq \mathbf 0}Q(\xi,\bk)\hat F
( \bk)e^{2\pi i( \bk, \bx)},
\ee
where
$$
Q(\xi,\bk):= Q_m(e^{2\pi i(\bk,\bx)},\xi).
$$
The following result is obtained in \cite{VT164}.

\begin{Theorem}\label{IT2} Let $1<p<\infty$ and $\|F\|_{p} \le 1$. Then there exists 
a set $\xi$ of $m$ points such that
$$
\|g_{\xi,Q,F}(\bx)\|_p \le Cp^{1/2} m^{-1/2}, \quad 2\le p<\infty, 
$$
$$
\|g_{\xi,Q,F}(\bx)\|_p \le C m^{\frac{1}{p}-1}, \quad 1< p <2.
$$
\end{Theorem}

We now formulate a corollary of Theorem \ref{IT1}, which complements Theorem \ref{IT2}. 
Let function $g_{\xi,Q,F}(\bx)$, associated with a function $F$ and a cubature formula $Q_m(\cdot,\xi)$, be defined by (\ref{I1}). 
\begin{Corollary}\label{IC1} Suppose that $\|F\|_\infty <\infty$ and function $F^0$ is a real 
even function, satisfying the condition $\hat F^0(\bk) \ge 0$, $\bk\in \Z^d$, $\bk\neq \mathbf 0$. Then for any $m\in\N$ there exists (provided by the Averaging Search algorithm applied to $g=F^0$) a set $\xi:=\{\xi^j\}_{j=1}^m$ of points in $[0,1)^d$ such that  
$$
\|g_{\xi,Q,F}(\bx)\|_\infty \le C\|F\|_\infty m^{-1/2}.
$$

\end{Corollary}

We note that the Averaging Search algorithm is not a greedy type algorithm. The following greedy version of this algorithm has been studied in a very recent paper \cite{BS}.

{\bf Greedy Averaging Search algorithm.} Let  $g$ be a real continuous $1$-periodic function satisfying  condition $\int_{[0,1)^d} g(\bx) d\bx =0$. We build a sequence $\xi^1$,..., $\xi^m$ of points from $[0,1)^d$ inductively. At the first step choose any $\xi^1\in [0,1)^d$. Suppose, $m\ge 2$ and after $m-1$ steps of the algorithm we have built points $\xi^1$,..., $\xi^{m-1}$. Then, at the $m$th step we choose $\xi^m\in [0,1)^d$ such that
\be\label{G1}
\sum_{j=1}^{m-1} g(\xi^m-\xi^j)  = \min_{\bx\in [0,1)^d} \sum_{j=1}^{m-1} g(\bx-\xi^j).
\ee

Clearly, the Greedy Averaging Search algorithm is a realization of the Averaging Search algorithm and, therefore, Theorem \ref{IT1} and Corollary \ref{IC1} hold for points obtained 
by the Greedy Averaging Search algorithm. It is an interesting open problem, which is discussed in detail in \cite{BS}, to understand if the Greedy Averaging Search algorithm can give better error bounds than $m^{-1/2}$. 

\section{Proof of Theorem \ref{asT1}} 
\label{P}

We begin with a simple identity, which was used in \cite{VT50} in a context of numerical integration (see also \cite{BS}). 

\begin{Lemma}\label{PL1} Let $g$ be a $1$-periodic function with absolutely convergent Fourier series satisfying  condition $\int_{[0,1)^d} g(\bx) d\bx =0$. For a given set of points $X_m:=\{\bx^j\}_{j=1}^m$,
denote
$$
  Q(X_m,\bk) := Q_m(e^{2\pi i(\bk,\bx)}, X_m)=\frac{1}{m} \sum_{j=1}^m e^{2\pi i (\bk,\bx^j)}.
$$
Then
$$
\sum_{j,n=1}^m g(\bx^n-\bx^j) = m^2\sum_{\bk\neq \mathbf 0} \hat g(\bk) |Q(X_m,\bk)|^2.
$$
\end{Lemma}
\begin{proof} For the reader's convenience we present this simple proof here. We have
$$
g(\bx^n-\bx^j) = \sum_{\bk\neq \mathbf 0} \hat g(\bk) e^{2\pi i (\bk,\bx^n)} e^{-2\pi i (\bk,\bx^j)}.
$$
Performing summation with respect to $n$ and $j$ we obtain the required identity. 
\end{proof}

We continue the proof of Theorem \ref{IT1}. By duality relation (see \cite{VT89} and \cite{VTbookMA}, p.254, (6.3.2)) we obtain
\be\label{P2}
Q_m(\bW^F_1,X_m) = \|\hat F(\mathbf 0) - \frac{1}{m} \sum_{\mu=1}^m F(\bx^\mu-\by)\|_\infty
= \|\frac{1}{m} \sum_{\mu=1}^m F^0(\bx^\mu-\by)\|_\infty.
\ee
We have
$$
\|\frac{1}{m} \sum_{\mu=1}^m F^0(\bx^\mu-\by)\|_\infty \le \sum_{\bk} |\hat F^0(\bk)||Q(X_m,\bk)|
$$
\be\label{P3}
  \le \left(\sum_{\bk} |\hat F^0(\bk)|\right)^{1/2} \left(\sum_{\bk} |\hat F^0(\bk)||Q(X_m,\bk)|^2\right)^{1/2}.
\ee
By our assumption on the $F^0$ we obtain
$$
\|F^0\|_\infty = F^0(\mathbf 0) = \sum_{\bk} |\hat F^0(\bk)|.
$$
Using Lemma \ref{PL1} we obtain from here and (\ref{P3})
\be\label{P4}
\|\frac{1}{m} \sum_{\mu=1}^m F^0(\bx^\mu-\by)\|_\infty \le \|F^0\|_\infty^{1/2} m^{-1} \left(\sum_{j,n=1}^m F^0(\bx^n-\bx^j)\right)^{1/2}.
\ee
We now set $\bx^j =\xi^j$ with $\xi^j$ obtained from the Averaging Search algorithm applied to $g=F^0$. Then, we have 
$$
\sum_{j,n=1}^m F^0(\bx^n-\bx^j) = mF^0(\mathbf 0) + 2\sum_{1\le j<n\le m}F^0(\xi^n-\xi^j) 
$$
$$
=
mF^0(\mathbf 0) + 2\sum_{n=2}^m\sum_{ j=1} ^{n-1}F^0(\xi^n-\xi^j).
$$
It remains to note that by the choice of $\xi^n$ we have 
$$
\sum_{ j=1} ^{n-1}F^0(\xi^n-\xi^j) \le 0.
$$

{\bf Acknowledgment.}   The work was supported by the Russian Federation Government Grant No. 14.W03.31.0031.

\end{document}